\documentclass[11pt, reqno]{amsart}

\usepackage{color}
\usepackage{amsmath}
\usepackage{latexsym,amssymb}
\usepackage{stmaryrd}
\usepackage{bm}
\usepackage{enumerate}
\usepackage{setspace}
\newtheorem{theorem}{\bf Theorem}[section]
\newtheorem{prop}{\bf Proposition}[section]

\numberwithin{equation}{section}

\begin{document}

\baselineskip=17pt

\title[]
{Continuous-time Zero-Sum Stochastic Game with Stopping and Control}

\author[Chandan Pal]{Chandan Pal}
\address{Department of Mathematics\\
Indian Institute of Technology Guwahati\\
Guwahati, Assam, India}
\email{cpal@iitg.ernet.in}

\author[Subhamay Saha]{Subhamay Saha}
\address{Department of Mathematics\\
Indian Institute of Technology Guwahati\\
Guwahati, Assam, India}
\email{saha.subhamay@iitg.ernet.in}


\date{}

\begin{abstract}
We consider a zero-sum stochastic game for continuous-time Markov chain with countable state space and unbounded transition and pay-off rates. The additional feature of the game is that the controllers together with taking actions are also allowed to stop the process. Under suitable hypothesis we show that the game has a value and it is the unique solution of certain dynamic programming inequalities with bilateral constraints. In the process we also prescribe a saddle point equilibrium. 
\vspace{2mm}

\end{abstract}

\maketitle

\section{INTRODUCTION}
In this article we consider a zero-sum stochastic game for continuous-time Markov chain. The transition and reward rates are assumed to be unbounded. The additional feature is that players other than taking actions also has the option of stopping the game. We show that the game has a value and that it is the unique solution of a set of dynamic programming inequalities with bilateral constrains. The existence of optimal strategies for both players is also established. These optimal strategies also give optimal stopping rules for both players. Stochastic control problems for continuous time Markov chains, both for one controller and multi-controller setup has been studied by a variety of authors, see \cite{c23, c20, c22, c21, c8} and references therein. Stochastic games with only stopping was introduced by Dynkin \cite{c3}. Such games also known in literature as Dynkin games has been investigated for discrete time case, see (\cite{c19, c10, c13, c18}) and references therein, as well as for continuous-time case, see (\cite{c2, c17, c16, c11, c14, c15}) and references therein. Stochastic games with control and stopping has been studied for discrete time case in \cite{c6}, and for continuous-time non-degenerate diffusion  in (\cite{c1, c5}). The authors in \cite{c24} consider a zero-sum stochastic game in a very general framework driven by Brownian motion and Poisson random measure. In their set-up the controllers are allowed to stop controlling the process at any time. However, the proof techniques of our paper are different and less technical as compared to \cite{c24}. Moreover, it is important to note that in \cite{c24} it is assumed that the diffusion co-efficient associated with the Brownian motion has to be invertible and hence can not be taken to be 0. Therefore, a wide class of stochastic game problems involving pure jump processes, for example controlled queues, can be much more suitably analysed using the tools and techniques of this paper as compared to \cite{c24}. We give an example of that nature in the last section. The rest of the paper is organized as follows. In section 2, we give the detailed problem formulation and in section 3 we prove the existence of value of the game and saddle point equilibrium. Finally, in section 4, we conclude with an illustrative example.

\section{Game Formulation}
The stochastic game model that we are interested in is given by $\{S,U,V,q,r,\psi_1,\psi_2\}$. The components have the following interpretation. $S$ is countable set and without any loss of generality we take $S=\{0,1,2,3, \cdots\}$. $S$ is the state space of the controlled continuous-time Markov chain. $U$ and $V$ are metric spaces representing the action sets of player $I$ and $II$ respectively. The component $q=[q(j|i,a,b)]$ is the controlled transition rate matrix, satisfying the following properties:

\noindent {\bf (i)} $ q(j|i,a,b)\geq 0 $ for all $i \in S $, $i\neq j$, $a \in U$, $b \in V$.

\noindent {\bf (ii)} It is assumed to be conservative, i.e., $$\sum_{j\in S}q(j|i,a,b)=0 \; \; \mbox{for all} \; i \in S, a\in U, b\in V .$$ 

 \noindent {\bf (iii)} 
We also assume it is stable, i.e., 
\[
q(i)=\sup_{a \in U, b \in V} \sum_{j\neq i}q(j|i,a,b)<\infty \mbox{ for all} \; i\in S \, .
\]

The reward rate is given by $r : S \times U \times V \to [0,\infty)$ and $\psi_i: S \rightarrow [0,\infty),\; i=1,2$ are the stopped pay-off functions for the players.
 
 At time $t=0$, we suppose that the  process starts from state $i$ and player $I$ and player $II$ independently chooses  actions $a$ and $b$ then player $II$ receives a reward at the rate $r(i,a,b)$ until the next jump epoch which occurs after an exponential ($\sum_{j\neq i}q(j|i,a,b)$) amount of time. The next state of the process is $j$ with probability $\dfrac{q(j|i,a,b)}{\sum_{j\neq i}q(j|i,a,b)}$. The game then repeats from the new sate $j$. If at state $i$, player $I$ decides to stop the game then, player $II$ receives a pay-off of $\psi_1(i)$, whereas if player $II$ decides to stop then she receives a pay-off equal to $\psi_2(i)$. Player $II$ tries to maximize her accumulated expected discounted reward, while player $I$ wishes to minimize the same. Here we will consider only randomized stationary control although things go through with randomized Markov control as well.
 
A randomized stationary control for player $I$ is a measurable function $\Phi: S \to \mathcal{P}(U)$. Similarly, $\Psi: S \to \mathcal{P}(V)$ is a randomized stationary control for player $II$. We denote by $\Pi^1$ and $\Pi^2$ the set of all randomized stationary controls for player $I$ and player $II$ respectively. In order to guarantee the existence of a non explosive process (finite jumps in a finite time) we assume the following:

\noindent {\bf Assumption (A1):}

\noindent There exists $N$ non-negative functions $w_n$ on $S$ and a positive constant $c$ such that for all $i \in S, a\in U, b\in V $,$$\sum_{j\in S}q(j|i,a,b)w_n(j)\leq w_{n+1}(i),\; \;  \mbox{for} \;  n=1,2,\cdots, N-1$$ and furthermore, $$\sum_{j\in S}q(j|i,a,b)w_N(j)\leq 0$$ and $$q(i)\leq c(w_1(i)+w_2(i)+\cdots+w_N(i)), \; \mbox{for all} \; i\in S.$$ 

It is well-known that under the above assumptions, for randomized stationary controls $(\Phi,\Psi)$  there exists a non explosive continuous time Markov chain, see \cite{c7}. We denote the state process by $X_t$ and let $U_t$ and $V_t$ denote the control processes for player $I$ and $II$ respectively.

Let $\{\mathcal{F}_t:t\geq 0\}$ denote the natural filtration of $\{X_t:t\geq 0\}$. Then a strategy for player $I$ is a pair $(\Phi,\Theta^1)$ where $\Phi \in \Pi^1$ and $\Theta^1$ is a $\mathcal{F}_t$-stopping time. Similarly, for player $II$ a strategy is a pair  $(\Psi,\Theta^2)$ where $\Psi \in \Pi^2$ and $\Theta^2$ is a $\mathcal{F}_t$-stopping time. The evaluation criterion is given by
\begin{align*}
J_{\alpha}(i,\Phi,\Psi,\Theta^1,\Theta^2)&=E_i^{\Phi,\Psi}\biggl[\int_0^{\Theta^1\wedge \Theta^2}e^{-\alpha t} r(X_t,U_t,V_t)dt\\&+e^{-\alpha(\Theta^1\wedge \Theta^2)}\bigl\{\psi_1(X_{\Theta^1})1_{\{\Theta_1<\Theta_2\}}+\psi_2(X_{\Theta^2})1_{\{\Theta_1\geq\Theta_2\}} \bigr\}\biggr],
\end{align*} where $\alpha >0$ is the discount factor, $1_{\{\cdot\}}$ is the indicator function and $E_i^{\Phi,\Psi}$ is the expectation operator with respect to the probability measure when the initial state is $i$ and player $I$ is using the control $\Phi$ and player $II$ is using the control $\Psi$.
Player $II$ wishes to maximize $J_{\alpha}(i,\Phi,\Psi,\Theta^1,\Theta^2)$ over her strategies $(\Psi,\Theta^2)$ and player $I$ wishes to minimize the same over all pairs $(\Phi,\Theta^1)$. Define 
\[
U(i)=\inf_{(\Phi,\Theta^1)}\sup_{(\Psi,\Theta^2)}J_{\alpha}(i,\Phi,\Psi,\Theta^1,\Theta^2)
\]
and 
\[
L(i)=\sup_{(\Psi,\Theta^2)}\inf_{(\Phi,\Theta^1)}J_{\alpha}(i,\Phi,\Psi,\Theta^1,\Theta^2).
\]
Then $U(i)$ is called the upper value of the game and $L(i)$ is called the lower value of the game. The game is said to have a value if $U(i)=L(i)$. 

A strategy $(\Phi^*, \Theta^{1*})$ is said to be optimal for player $I$ if 
\[
L(i)\geq J_{\alpha}(i,\Phi^*,\Psi,\Theta^{1*},\Theta^2) \; \mbox{for all} \; i \in S
\]
and for all strategies $(\Psi,\Theta^2)$ of player $II$. Analogously, A strategy $(\Psi^*, \Theta^{2*})$ is said to be optimal for player $II$ if 
\[
U(i)\leq  J_{\alpha}(i,\Phi,\Psi^*,\Theta^{1},\Theta^{2*}) \; \mbox{for all} \; i \in S
\]
and for all strategies $(\Phi,\Theta^1)$ of player $I$. $((\Phi^*, \Theta^{1*}),(\Psi^*, \Theta^{2*}))$ is called a saddle point equilibrium, if it exists.
\section{Existence of Value and Saddle Point Equilibrium}
In order to characterize the value of the game and to establish the existence of a saddle point equilibrium we will need the following assumption:

\noindent {\bf Assumption (A2):} 

\noindent {\bf (i)} $ U $ and $V$ are compact sets; 

\noindent {\bf (ii)} $r(i,a,b)$ and $ q(j|i,a,b)$ are continuous in  $(a,b)\in U \times V $;

\noindent {\bf (iii)} Let $W(i)=w_1(i)+w_2(i)+\cdots+w_N(i), \; \mbox{for all} \; i\in S.$ The function $\sum_{j\in S}q(j|i,a,b)W(j)$ is continuous in $(a,b)\in U \times V $; 

\noindent {\bf (iv)} there is a constant $M$ such that 
\[
r(i,a,b)\leq M W(i),  \; \mbox{for all} \; \; i\in S \; \mbox{and} \; (a,b)\in U \times V, 
\]
\[
\psi_1(i)\leq M W(i),  \; \mbox{for all} \; \; i\in S
\]
and 
\[
\psi_2(i)< \psi_1(i),  \; \mbox{for all} \; \; i\in S;
\]
\noindent {\bf (v)} there exists a non-negative function $\tilde{W}$ on $S$ and positive constants $ c$ and $\tilde{c}$ such that 
\[
q(i)W(i)\leq M \tilde{W}(i), \; \; \mbox{for all} \; \; i\in S
\]
and
\begin{align*}
 &\sum_{j\in S}q(j|i,a,b)\tilde{W}(j)\leq c\tilde{W}(i)+\tilde{c} \; \; \mbox{for all} \; \; i\in S\;\; \mbox{and} \;\; (a,b)\in U \times V.
\end{align*}
Set
\[
 B_W(S) \ = \ \biggl\{ f : S \to [0,\infty) \big| \sup_{i \in S} \frac{f(i)}{W(i)} < \infty \biggr\},
\]
where $W$ is as in \textbf{(A2)}.
 Define for $f\in B_W(S)$,
\begin{equation*}\label{vn}
\|f\|_W \ = \ \sup_{i \in S} \frac{f(i)}{W(i)} \, .
\end{equation*}
 Then
$B_W(S)$ is a Banach space with the norm $\|\cdot\|_W $.  \\
For any two states $i,j \in S$, any two probability measures $\mu \in \mathcal{P}(U) $ and $\nu \in \mathcal{P}(V)$ define
\[
\tilde{r}(i,\mu,\nu) = \int_{V}\int_{U}r(i,a,b)\mu(da)\nu(db)
\] 
and
\[
\tilde{q}(j|i,\mu,\nu) = \int_{V}\int_{U}q(j|i,a,b)\mu(da)\nu(db).
\] 
For $\phi \in B_W(S)$ define
\begin{align*}
H_{\alpha}^+(i,\phi)= \inf_{\mu \in \mathcal{P}(U)}\sup_{\nu\in \mathcal{P}(V)} \Big [ \tilde{r}(i,\mu,\nu) 
+  \sum_{j\in S}\tilde{q}(j|i,\mu,\nu)\phi(j) \Big ]  
\end{align*}
and
\begin{align*}
H_{\alpha}^-(i,\phi)= \sup_{\nu\in \mathcal{P}(V)}  \inf_{\mu \in \mathcal{P}(U)}\Big [ \tilde{r}(i,\mu,\nu) 
+  \sum_{j\in S}\tilde{q}(j|i,\mu,\nu)\phi(j) \Big ].
\end{align*}
Further define
\begin{align*}
I_{\alpha}^+(i,\phi)= \inf_{\mu \in \mathcal{P}(U)}\sup_{\nu\in \mathcal{P}(V)} \Big [\dfrac{ \tilde{r}(i,\mu,\nu)}{\alpha+q(i)+1} 
+ \dfrac{ q(i)+1}{\alpha+q(i)+1} \sum_{j\in S}\tilde{p}(j|i,\mu,\nu)\phi(j) \Big ] , 
\end{align*}
where $$\tilde{p}(j|i,\mu,\nu)=\dfrac{\tilde{q}(j|i,\mu,\nu)}{ q(i)+1}+\delta_{ij}$$ ($\delta_{ij}$ is the Kronecker delta). Similarly define,
\begin{align*}
I_{\alpha}^-(i,\phi)= \sup_{\nu\in \mathcal{P}(V)} \inf_{\mu \in \mathcal{P}(U)} \Big [\dfrac{ \tilde{r}(i,\mu,\nu)}{\alpha+q(i)+1}
+ \dfrac{ q(i)+1}{\alpha+q(i)+1} \sum_{j\in S}\tilde{p}(j|i,\mu,\nu)\phi(j) \Big ] . 
\end{align*}
Note that by Fan's minimax theorem \cite{c4}, $H_{\alpha}^+=H_{\alpha}^-:=H_{\alpha}$ and $I_{\alpha}^+=I_{\alpha}^-:=I_{\alpha}$. Now consider the following dynamic programming inequalities with bilateral constraints:
\begin{align}\label{dpi}
&\psi_2(i)\leq \phi(i) \leq \psi_1(i)&& \nonumber \\
&\alpha \phi(i) -H_{\alpha}(i,\phi)= 0 && {\rm if} \; \; \psi_2(i)< \phi(i) <\psi_1(i), \nonumber \\
&\alpha \phi(i) -H_{\alpha}(i,\phi)\geq 0 &&  {\rm if} \; \; \psi_2(i)= \phi(i), \\
&\alpha \phi(i) -H_{\alpha}(i,\phi)\leq 0 &&  {\rm if} \; \; \psi_1(i)= \phi(i). \nonumber 
\end{align} Now, 
\begin{align*}
&\alpha \phi(i) -H_{\alpha}(i,\phi)= 0\\
\Longleftrightarrow \, &\alpha \phi(i)=H_{\alpha}(i,\phi)\\
\Longleftrightarrow \, &(\alpha+q(i)+1)\phi(i)= \inf_{\mu \in \mathcal{P}(U)}\sup_{\nu\in \mathcal{P}(V)} \Big [ \tilde{r}(i,\mu,\nu) 
+  \sum_{j\in S}\tilde{q}(j|i,\mu,\nu)\phi(j)+(q(i)+1)\phi(i) \Big ]\\
\Longleftrightarrow \, &\phi(i)=I_{\alpha}(i,\phi)\,.
\end{align*}
Similarly, for the inequalities. Thus, \eqref{dpi} is equivalent to
\begin{align}\label{DPI}
 &\psi_2(i)\leq \phi(i) \leq \psi_1(i)&& \nonumber \\
 &\phi(i) -I_{\alpha}(i,\phi)= 0 && {\rm if} \; \; \psi_2(i)< \phi(i) <\psi_1(i) \nonumber \\
&\phi(i) -I_{\alpha}(i,\phi)\geq 0 &&  {\rm if} \; \; \psi_2(i)= \phi(i) \\
 &\phi(i) -I_{\alpha}(i,\phi)\leq 0 &&  {\rm if} \; \; \psi_1(i)= \phi(i). \nonumber 
\end{align}
\begin{prop}\label{equiv}
Under assumptions \textbf{(A1)} and \textbf{(A2)}, the following are equivalent.
\begin{itemize}
\item[(i)] $\phi$ satisfies (\ref{DPI}).
\item[(ii)] $\phi(i)=\min\{\max\{I_{\alpha}(i,\phi);\psi_2(i)\};\psi_1(i)\}$.
\item[(iii)] $\phi(i)=\max\{\min\{I_{\alpha}(i,\phi);\psi_1(i)\};\psi_2(i)\}$.
\end{itemize}
\end{prop}
\begin{proof}
Here we prove only the equivalence of (i) and (ii), others can be proved similarly. Suppose (i) is true and $i\in S$ is such that
$ \psi_2(i)< \phi(i) <\psi_1(i)$. Then 
\begin{align*}
 \phi(i) &= I_{\alpha}(i,\phi)\\
 &= \max\{I_{\alpha}(i,\phi);\psi_2(i)\}\\
&= \min\{\max\{I_{\alpha}(i,\phi);\psi_2(i)\};\psi_1(i)\} \nonumber 
\end{align*}
If $ \psi_2(i)= \phi(i)$. Then $\phi(i) \geq I_{\alpha}(i,\phi)$. Therefore
\begin{align*}
 \phi(i)  &= \max\{I_{\alpha}(i,\phi);\psi_2(i)\}\\
&= \min\{\max\{I_{\alpha}(i,\phi);\psi_2(i)\};\psi_1(i)\}, \nonumber 
\end{align*}
since $\psi_2(i)\leq \phi(i) \leq \psi_1(i)$.

If $ \psi_1(i)= \phi(i)$, then $\phi(i) \leq I_{\alpha}(i,\phi)$ and $\psi_2(i)\leq \phi(i) \leq \psi_1(i)$. Hence
\begin{align*}
 I_{\alpha}(i,\phi) = \max\{I_{\alpha}(i,\phi);\psi_2(i)\}.
\end{align*}
Therefore
\begin{align*}
\phi(i)=\psi_1(i)= \min\{\max\{I_{\alpha}(i,\phi);\psi_2(i)\};\psi_1(i)\}.  
\end{align*}
Now assume that (ii) is true, i.e.,
\begin{align*}
\phi(i)= \min\{\max\{I_{\alpha}(i,\phi);\psi_2(i)\};\psi_1(i)\}.  
\end{align*}
Suppose $i\in S$ is such that $ \psi_2(i)< \phi(i) <\psi_1(i)$. Then 
\begin{align*}
 \phi(i)  &= \max\{I_{\alpha}(i,\phi);\psi_2(i)\} \\
 &= I_{\alpha}(i,\phi)
\end{align*}
If $ \psi_2(i)= \phi(i)$. Then by assumption \textbf{(A2)}, $\phi(i) < \psi_1(i)$. Therefore
\begin{align*}
 \phi(i)  &= \max\{I_{\alpha}(i,\Phi);\psi_2(i)\}\\
&\geq I_{\alpha}(i,\phi).
\end{align*}
If $ \psi_1(i)= \phi(i)$. Then by assumption \textbf{(A2)}, $\phi(i) > \psi_2(i)$. Therefore $\phi(i)=\psi_1(i)  \leq \max\{I_{\alpha}(i,\phi);\psi_2(i)\}$, which implies that  $ \phi(i) \leq  I_{\alpha}(i,\phi)$ (since  $\phi(i) > \psi_2(i)$ ). It is easy to see that $\psi_2(i)\leq \phi(i) \leq \psi_1(i)$. Hence, $\phi$ satisfies (\ref{DPI}).
\end{proof}
Now define the operator $T:B_W(S) \to B_W(S) $ by
\begin{align*}
T\phi(i):= \min\{\max\{I_{\alpha}(i,\phi);\psi_2(i)\};\psi_1(i)\}.  
\end{align*}
Let $u_0(i)=\psi_2(i)$ and $u_n=T u_{n-1},\; n \geq 1$. Then the following is true.
\begin{prop}\label{fp}
Under assumptions \textbf{(A1)} and \textbf{(A2)}, the sequence of functions $\{u_n\}_{n\geq 0}$ is a non-decreasing and there exists $u^* \in B_W(S)$ such that $\displaystyle{\lim_{n \to \infty}}u_n=u^*$. Further $u^*$ is a fixed point of $T$, i.e., $T u^*=u^*$.
\end{prop}
\begin{proof}
Clearly $u_1(i)=T u_0(i)\geq \psi_2(i)=u_0(i)$ for all $i\in S$.
Now suppose  $u_n \geq u_{n-1}$. It is easy to see that $I_{\alpha}(i,\Phi) $ is monotone in $\Phi$. Thus, we have 
\begin{align*}
u_{n+1}(i)=T u_{n}(i) &= \min\{\max\{I_{\alpha}(i,u_n);\psi_2(i)\};\psi_1(i)\} \\
&\geq  \min\{\max\{I_{\alpha}(i,u_{n-1});\psi_2(i)\};\psi_1(i)\} \\
&= T u_{n-1}(i)= u_{n}(i).
\end{align*}
Thus, by induction we have that $\{u_n\}_{n\geq 0}$ is a non-decreasing. Therefore there exists $u^* \in B_W(S)$ such that $u^*(i)= \displaystyle{\lim_{n \to \infty}}u_n(i)$ for all $i \in S$. Now clearly $$T u^{*}(i)\geq T u_{n}(i)=u_{n+1}(i).$$ Taking limit $n \to \infty$ on both sides we get,
$$T u^{*}(i)\geq u^{*}(i)\; \mbox{for all} \; i \in S.$$
Now for the reverse inequality, 
\begin{align*}
 &u_{n+1}(i)=T u_{n}(i) \\
 &= \min  \{\max  \{I_{\alpha}(i,u_n);\psi_2(i)\};\psi_1(i)\} \\
&= \min \Big \{\max \Big \{\sup_{\nu\in \mathcal{P}(V)} \inf_{\mu \in \mathcal{P}(U)} \Big [\dfrac{ \tilde{r}(i,\mu,\nu)}{\alpha+q(i)+1}
+ \dfrac{ q(i)+1}{\alpha+q(i)+1} \sum_{j\in S}\tilde{p}(j|i,\mu,\nu)u_n(j) \Big ];\psi_2(i)\Big \};\psi_1(i)\Big \} \\
 &\geq \min\Big\{\max\Big\{ \Big [\dfrac{ \tilde{r}(i,\mu_n^*,\nu)}{\alpha+q(i)+1} + \dfrac{ q(i)+1}{\alpha+q(i)+1} \sum_{j\in S}\tilde{p}(j|i,\mu_n^*,\nu)u_n(j) \Big ];\psi_2(i)\Big \};\psi_1(i)\Big \},
\end{align*}
where $\nu \in \mathcal{P}(V)$ is arbitrary. The existence of $\mu_n^* \in \mathcal{P}(U)$ is ensured by assumption \textbf{(A2)}. Now since $\mathcal{P}(U)$ is compact, there exists $\mu^* \in \mathcal{P}(U)$ and a subsequence of $\{\mu_n\}$ converging to $\mu^* $ in  $\mathcal{P}(U)$. Thus, by an extension of Fatou's lemma [Lemma 8.3.7(b) in \cite{c9}] we have by letting $n \to \infty$ on both sides,
\begin{align*}
 u^{*}(i) &\geq \min \Big \{\max \Big \{ \Big [\dfrac{ \tilde{r}(i,\mu^*,\nu)}{\alpha+q(i)+1} 
+ \dfrac{ q(i)+1}{\alpha+q(i)+1} \sum_{j\in S}\tilde{p}(j|i,\mu^*,\nu)u^*(j) \Big ];\psi_2(i)\Big \};\psi_1(i)\Big \} \\
&\geq \min \Big \{\max \Big \{ \inf_{\mu \in \mathcal{P}(U)} \Big [\dfrac{ \tilde{r}(i,\mu,\nu)}{\alpha+q(i)+1}
+ \dfrac{ q(i)+1}{\alpha+q(i)+1} \sum_{j\in S}\tilde{p}(j|i,\mu,\nu)u^*(j) \Big ];\psi_2(i)\Big \};\psi_1(i)\Big \}.
\end{align*}
Since the above is true for any $\nu \in \mathcal{P}(V)$. Hence, we have
\begin{eqnarray*}
 u^{*}(i) &\geq  & \min  \{\max  \{I_{\alpha}(i,u^*);\psi_2(i)\};\psi_1(i)\} \\
&=& T u^*(i).
\end{eqnarray*}
Thus, we are done.
\end{proof}
Let $\Phi^* \in \Pi^1$ and  $\Psi^* \in \Pi^2$ be such that
{\small\begin{align*}
H_{\alpha}(i,u^*)= \sup_{\nu\in \mathcal{P}(V)} \Big [ \tilde{r}(i,\Phi^*(i),\nu) 
+  \sum_{j\in S}\tilde{q}(j|i,\Phi^*(i),\nu)u^*(j) \Big ]  
\end{align*}}
and
{\small\begin{align*}
H_{\alpha}(i,u^*)= \inf_{\mu \in \mathcal{P}(U)}\Big [ \tilde{r}(i,\mu,\Psi^*(i)) 
+  \sum_{j\in S}\tilde{q}(j|i,\mu,\Psi^*(i))u^*(j) \Big ].
\end{align*}}
 The existence of $\Phi^* \in \Pi^1$ and  $\Psi^* \in \Pi^2$ follows from assumption \textbf{(A2)} and a measurable selection theorem \cite{c12}.
 
 Define $$ A_1=\{i\in S| u^*(i)=\psi_1(i)\}$$ and $$ A_2=\{i\in S| u^*(i)=\psi_2(i)\}.$$ Let $\{X_t; t\geq 0\}$ be the state process governed by the stationary controls   $\Phi^* \in \Pi^1$ and  $\Psi^* \in \Pi^2$. Let $$ \Theta^{1*}=\inf \{t\geq 0 | X_t \in A_1\}$$ and $$ \Theta^{2*}=\inf \{t\geq 0 | X_t \in A_2\}.$$ Then we have our main theorem.
 \begin{theorem}
Assume \textbf{(A1)} and \textbf{(A2)}. Let $u^*$ be as in Proposition \ref{fp}. Then the stochastic game with stopping and control has a value and $u^*(i)=U(i)=L(i)$. Thus, $u^*$ is the unique fixed point of $T$. Further $((\Phi^*, \Theta^{1*}),(\Psi^*, \Theta^{2*}))$ is a saddle point equilibrium.
 \end{theorem}
 \begin{proof}
 Let $i \in S$ be such that $ u^*(i)<\psi_1(i)$. Let $\Psi$ be any stationary control of player $II$. Let $\{\tilde{X}_t; t\geq 0\}$ be the process governed by the stationary controls $\Phi^*$ and $\Psi $ and let $ \tilde{\Theta}^1=\inf \{t\geq 0 | \tilde{X}_t \in A_1\}$, $\Theta^2$ be any stopping time of player $II$. Then by Dynkin's formula we get for $T\geq 0$,
 \begin{align*}
& E_i^{\Phi^*,\Psi}\Big [e^{-\alpha (T \wedge \tilde{\Theta}^1 \wedge {\Theta}^2 )} u^*(\tilde{X}_{T \wedge \tilde{\Theta}^1 \wedge {\Theta}^2}) \Big ]-u^*(i) \\
&= E_i^{\Phi^*,\Psi}\int_0^{T \wedge \tilde{\Theta}^1 \wedge {\Theta}^2}e^{-\alpha t} \Big [-\alpha u^*(\tilde{X}_{t}) +  \sum_{j\in S}\tilde{q}(j|\tilde{X}_{t},\Phi^*(\tilde{X}_{t}),\Psi(\tilde{X}_{t}))u^*(j) \Big ]dt \\
 &\leq - E_i^{\Phi^*,\Psi}\int_0^{T \wedge \tilde{\Theta}^1 \wedge {\Theta}^2}e^{-\alpha t} r(\tilde{X}_{t}, U_t, V_t) dt.
\end{align*}
The last inequality is a consequence of the facts that $u^*$ is a fixed point of $T$, Proposition \ref{equiv} and equivalence of \eqref{DPI} and \eqref{dpi}.
Now letting $T \to \infty $ we get,
 \begin{align*}
u^*(i) &\geq E_i^{\Phi^*,\Psi}\int_0^{ \tilde{\Theta}^1 \wedge {\Theta}^2}e^{-\alpha t} r(\tilde{X}_{t}, U_t, V_t) dt\\&\qquad\qquad+ E_i^{\Phi^*,\Psi}\Big [e^{-\alpha (\tilde{\Theta}^1 \wedge {\Theta}^2 )} u^*(\tilde{X}_{ \tilde{\Theta}^1 \wedge {\Theta}^2}) \Big ]\\
&\geq J_{\alpha}(i,\Phi^*,\Psi,\tilde{\Theta}^1,\Theta^2).
\end{align*}
Since the above is true for any strategy $(\Psi, \Theta^2)$ of player $II$, we obtain $u^*(i)\geq U(i)$. Analogously it can be shown that $u^*(i)\leq L(i)$. Thus, we get, $U(i)\leq u^*(i)\leq L(i)$. On the other hand, we have trivially, $L(i)\leq U(i)$. Thus, $u^*(i)=U(i)=L(i)$. The uniqueness of the fixed point follows because we have just shown that any fixed point is the value of the game. Hence, we are done.
\end{proof}
\section{Example} Consider a single server queueing system having original arrival rate $\lambda(\cdot)$ and service rate $\mu(\cdot)$. Suppose that there are two parties or players. Depending on the number of people in the system, which is defined to be the state of the system, player I can modify the service rate by choosing some action $a$, which will result in an increased service rate equal to $\mu(i)+h(a)$, with $h$ being a function on the action space of player I. But this action will also result in a cost rate given by $c_1(i,a)$, if $i$ is the state of the system. On the other hand, player II can modify the arrival rate by choosing some action $b$, which will result in an increased arrival rate given by $\lambda(i)+g(b)$, with $g$ being a function on the action space of player II. The action of player II results in a cost rate given by $c_2(i,b)$. If at any given time there are $i$ customers in the system, then it generates a reward at the rate $c+ri$ for player II. Moreover, at any given time both the players have the option of quiting the system. If player I decides to quit when the state of the system is $i$, then player II gets a terminal reward equal to $\bar{c}+R(i)$ for some function $R(\cdot)$, whereas if player II decides to quit then she receives a terminal reward equal to $c^{\prime}$, with $c^{\prime}<\bar{c}$. Now this system can easily be modelled via the game model considered in this paper. So the transition rates are given by
$q(1|0,a,b)= \lambda(0)+g(b)=-q(0|0,a,b)$, 0 otherwise. For $i\neq 0$,
\begin{align*}
q(j|i,a,b)=\begin{cases}\lambda(i)+g(b)\quad&\mbox{for}\,\,j=i+1,\\
\mu(i)+h(a)\quad&\mbox{for}\,\,j=i-1,\\
-(\lambda(i)+g(b)+\mu(i)+h(a))\quad&\mbox{for}\,\,j=i,\\
0\quad&\mbox{otherwise}\,.
\end{cases} 
\end{align*}The reward rate is given by $r(i,a,b)=c+ri+c_1(i,a)-c_2(i,b)$. The terminal cost functions are given by, $\psi_1(i)=\bar{c}+R(i)$ and $\psi_2(i)=c^{\prime}$. Thus, the above example highlights the importance of the model considered in this paper.


\end{document}